\documentclass[10pt]{amsart}

\usepackage{amsmath, amsfonts, amssymb, xypic}
\usepackage{latexsym, verbatim}
\usepackage{graphicx}
\usepackage{color}
\usepackage{url}
\usepackage[table]{xcolor}
\usepackage{hyperref}
\usepackage{amsmath}

\newtheorem{thm}{Theorem}[section]
 \newtheorem{cor}[thm]{Corollary}
 
 \newtheorem{prop}[thm]{Proposition}

\theoremstyle{definition}

\newcommand{\Z}{\mathbb{Z}}

\newcommand{\CC}{\mathbb{C}}

\newcommand{\cA}{{\mathcal A}}

\newcommand{\la}{\lambda}
\newcommand{\La}{\Lambda}

\newcommand{\spa}{{\rm span}}

\newcommand{\Ke}{{\rm Ker}\,}

\newcommand{\del}{\partial}
\newcommand{\delb}{{\bar \partial}}
\newcommand{\mub}{{\bar \mu}}
\newcommand{\lab}{{\bar \lambda}}
\newcommand{\taub}{{\bar \tau}}

\newcommand{\Img}{\mathrm{Im}}

\title{Harmonic symmetries for Hermitian manifolds}

\author[S. Wilson]{Scott O. Wilson}
  \address[S. Wilson]{Department of Mathematics, Queens College, City University of New York, 65-30 Kissena Blvd., Flushing, NY 11367}
  \email{scott.wilson@qc.cuny.edu}

\thanks{The author acknowledges support provided by a PSC-CUNY Award, jointly funded by The Professional Staff Congress and The City University of New York.}


\begin{document}

\begin{abstract} Complex manifolds with compatible metric have a naturally defined subspace of harmonic differential forms that satisfy Serre, Hodge, and conjugation duality, as well as hard Lefschetz duality. This last property follows from a representation of $\mathfrak{sl}(2,\CC)$, generalizing the well known structure on the harmonic forms of compact K\"ahler manifolds. Some topological implications are deduced.
\end{abstract}

\maketitle


\section{Introduction}

While there are many known topological obstructions to a manifold having a K\"ahler metric, beyond complex dimension $2$, our understanding of the relationship between complex structures and topology is much more primitive.

Many interesting and important invariants of complex structures, including the dimensions of Dolbeault and Bott-Chern/Aeppli cohomologies, are bounded from below by topological quantities, such as Betti numbers. But there are comparatively few reverse inequalities, and these are desirable for showing that a complex manifold has non-trivial topology, or conversely that a smooth manifold does not have a complex structure.

This short note describes some topological and geometric inequalities for Hermitian manifolds, i.e. complex manifolds with compatible metric. They are expressed in terms of the kernel of a certain Laplacian-type operator, derived in some sense from the locus of the symplectic condition. 
The kernel of this second-order self-adjoint elliptic operator determines a subspace of harmonic forms that satisfies the Serre, Hodge, and conjugation dualities, generalizing the K\"ahler case. Moreover, there is an induced representation of $\mathfrak{sl}(2,\CC)$ on these harmonic forms, yielding a generalization of hard Lefschetz duality. In particular, it holds that the dimension of kernel of this elliptic operator, beginning from a given bidegree, is non-decreasing up to half the dimension of the manifold, as in the K\"ahler case, c.f. Corollary \ref{cor:nondec}. Several corollaries involving the Betti numbers are deduced from this.

The work here relies on a  generalization of the K\"ahler identities to the Hermitian setting, due to Demailly \cite{De}, but implicit in \cite{Gr}. The author was influenced by \cite{Po}, where similar Laplacian-type operators are used to study the Fr\"olicher spectral sequence. Finally, one may notice the results here mirror the structure and statements on the almost-K\"ahler side, recently described in \cite{CWAK}.  

The author thanks Joana Cirici for many discussions related to this work, as well as comments on an earlier draft.  

\section{Hermitian Identities}
Consider a complex manifold $(M,J)$ with Hermitian metric $\langle \, , \, \rangle$, and $(1,1)$-form $\omega$.
Let $L: \cA^{p,q} \to \cA^{p+1,q+1}$ be the Lefschetz operator on $(p,q)$-forms, $L(\eta) = \omega \wedge \eta$, and let 
$\La : \cA^{p,q} \to \cA^{p-1,q-1}$ be the adjoint.  The operators $\{L, \La, H = [L,\La]\}$ generate a representation of $\mathfrak{sl}(2,\CC)$ on $\cA = \bigoplus_{p,q} \cA^{p,q}$, \cite{Weil}. 

Let
 $\la = [\del, L] = (\del \omega \wedge \cdot )$, so that  
$\lab = [\delb,L] = (\delb \omega \wedge \cdot)$.
Then
\[
\left( d + (d \omega \wedge \cdot) \right) ^2 = (d + \la + \lab)^2 =0,
\]
which is equivalent to
\begin{equation} \label{eq:dlalab}
\begin{split}
\delb^2 &= 0 \\
\del^2 &=0 \\
[\delb,\del] &=0
\end{split}
\quad \quad
\begin{split}
\lab ^2 &=0 \\
[\delb, \lab] &= 0\\
[\lab, \del] + [\delb, \la]&=0
\end{split}
\quad \quad
\begin{split}
[\lab, \la] &=0\\
 [\del, \la] &=0 \\
\la ^2&=0,
\end{split}
\end{equation}
as well as all of the adjoint equations. 

The operator $\lab$ has bidegree $(1,2)$ and governs the symplectic condition: a Hermitian manifold is K\"ahler if and only if $\lab = 0$. By comparison, the component of exterior $d$ on an  almost complex manifold 
that governs integrability has mirror bidegree  $(-1,2)$, under the bigrading symmetry $(p,q) \to (n-p,q)$.

In \cite{De}, Demailly derives a set of Hermitian identities which generalize the K\"ahler identities.
Consider the zero-order \emph{torsion} operator $\tau:= [\La, \la]$ of bidegree $(1,0)$. Demailly shows

\begin{equation} \label{eq:Ldtau}
\begin{split}
[\La, \delb] &= -i(\del^* + \tau^*) \\
[\La, \del] &= i(\delb^* + \taub^*) \\
[L, \delb^*] &= -i(\del + \tau) \\
[L, \del^*] &= i(\delb+ \taub), \\
\end{split}
\end{equation}
with K\"ahler identities recovered in the case $\tau=0$. It is also shown that   $[\La, \tau ] = -2i \taub^*$ and $[L,\tau]= 3 \la$, so that
\begin{equation} \label{eq:Ltaula}
\begin{split}
[\La, \tau ] &= -2i \taub^* \\
[\La, \taub ] &= 2i \tau^* \\
[L, \tau^* ] &= -2i \taub \\
[L, \taub^* ] &= 2i \tau \\
\end{split}
\quad \quad
\begin{split}
[L,\taub ]  &= 3 \lab \\
[L,\tau ]  &= 3 \la \\
[\La,\taub^* ]  &= -3 \lab^* \\
[\La,\tau^* ]  &= -3 \la^*,
\end{split}
\end{equation}
and we also have
\begin{equation} \label{eq:Llatau}
\begin{split}
[\La, \la ] &=  \tau\\
[\La, \lab ] &=  \taub \\
[L, \la^* ] &=  -\tau^* \\
[L, \lab^* ] &= -\taub^* \\
\end{split}
\quad \quad
\begin{split}
[L,\la ]  &= 0 \\
[L,\lab ]  &= 0  \\
[\La, \la^* ]  &= 0 \\
[\La,\lab^* ]  &= 0.
\end{split}
\end{equation}

For any operator $\delta$, let $\Delta_\delta = [\delta, \delta^*]$, and let $\Delta_\delta^{p,q}$ denote the restriction to $\cA^{p,q}$.  It follows from
\eqref{eq:Ltaula}  and \eqref{eq:Llatau} that

\begin{cor}
For any Hermitian manifold
there is an induced representation of $\mathfrak{sl}(2,\CC) = \spa_\CC \{L, \La, H\}$ on the space
$\Ke (\Delta_\tau + \Delta_\taub + \Delta_\la + \Delta_\lab)$.
In fact,
\begin{equation} 
\begin{split}
[L,\Delta_\la] + [\la,\tau^*] &=0, 
\end{split}
\quad \quad
\begin{split}
[L,\Delta_\tau] &= 3 [\la, \tau^*]-2i[\tau,\taub].
\end{split}
\end{equation}
\end{cor}

\begin{proof} Using \eqref{eq:Llatau}
\[
[L,\Delta_\la] = [L,[\la,\la^*]] = [\la, [L,\la^*]] = -[\la,\tau^*]
\]
and
\[
[L,\Delta_\tau] = [L, [\tau, \tau^*]] = [[L,\tau],\tau^*] + [\tau,[L,\tau^*]]
= 3 [\la, \tau^*]-2i[\tau,\taub].
\]
\end{proof}

The following additional relations can also be deduced, by the Jacobi identity, 
\begin{equation} 
\begin{split}
[\lab, \tau] 
&= -[\taub,\la] ,
\end{split}
\quad \quad
\begin{split}
[\tau,\tau] &= 2i [ \la ,\taub^*],
\end{split}
\quad \quad
\begin{split}
2 [\taub^*,\tau] &= 3 [\lab^*,\la],
\end{split}
\end{equation}
as well as two more relations
\begin{equation} 
\begin{split}
[\del , \delb^* + \taub^*] = 0,
\end{split}
\quad \textrm{and} \quad
\begin{split}
[\delb, \del^* + \tau^*]= 0,
\end{split}
\end{equation}
that are proved in \cite{De}.

From the above relations we obtain, by taking conjugates and adjoints, all quadratic relations on the $\Z_2$-graded Lie algebra generated by the operators $\del, \delb, \tau, \taub, \la, \lab, L$ and their adjoints. In the physics literature, it is often lamented that the failure of the K\"ahler condition corresponds to the symmetry breaking in the algebra of operators. Contrary their being any deficiency in the algebra, and similar to the almost-K\"ahler case \cite{CWAK}, this algebra of operators can be well understood by including the lower order terms that vanish in the K\"ahler case. Moreover, there seems to be some interesting relationship between the two discussions.

\section{Harmonic Symmetries}

Let $(M,J, \omega)$ be a compact Hermitian manifold. Consider the following positive definite self-adjoint elliptic operator of order two:
\[
\square = \Delta_\del + \Delta_\delb + \Delta_\tau + \Delta_\taub + \Delta_\la + \Delta_\lab .
\]
This is a real operator, and the last four summands are all order zero. Let $\square^{p,q}$ denote the restriction to $\cA^{p,q}$. 

\begin{thm} \label{main}
Let $(M,J, \omega)$ be a compact Hermitian manifold of complex dimension $n$.  For any $0 \leq k \leq 2n$,  there is an orthogonal direct sum decomposition
 \[
  \Ke(\square) \cap \cA^{k} = \bigoplus_{p+q =k}  \Ke(\square^{p,q}). 
   \]
For all $0 \leq p,q \leq  n$ the following dualities hold:
 \begin{enumerate}
  \item \emph{(Complex conjugation)}. We have equalities 
    \[   \Ke(\square^{p,q})   =   \overline{ \Ke(\square^{q,p}) } .\]
\item \emph{(Hodge duality)}. The Hodge $\star$-operator induces isomorphisms
\[
\star :  \Ke(\square^{p,q})  \to  \Ke(\square^{n-q,n-p}).
\]
\item \emph{(Serre duality)}. There are isomorphisms
\[ \Ke(\square^{p,q})  \cong  \Ke(\square^{n-p,n-q}).\]
 \end{enumerate}

The operators $\{L,\Lambda, H \}$ define a finite dimensional representation of $\mathfrak{sl}(2,\CC)$ on  
\[
 \Ke(\square)  = \bigoplus_{p,q \geq 0}  \Ke(\square^{p,q}).
\] 
Moreover, for every $0 \leq p \leq n$ and all $p \leq k \leq n$ the maps
\[
L^{n-k} :  
 \Ke(\square^{p,k-p})  \stackrel{\cong}{\longrightarrow} 
 \Ke(\square^{p+n-k,n-p}).
\]
are isomorphisms.
\end{thm}

\begin{proof}
For the first claim, $\Ke(\square)$ is the intersection of the kernels of $\delb, \del, \tau, \taub, \la, \lab$ and their adjoints, as can be seen by expanding $\langle \square \, \eta, \eta\rangle$. Any $k$-form in the kernel of one of these operators must have all of its $(p,q)$ components in the kernel as well, so the first claim follows. 

The operator satisfies $\overline{\square} = \square$, and one can check that for any $\delta = \delb, \del, \tau, \taub, \la, \lab$, we have $\delta^* = - \star \overline \delta \star$. This implies for any such $\delta$ that  $\Delta_\delta  \, \star = - \star \Delta_{\overline{\delta}}$, showing claims $(1)$, $(2)$, and $(3)$. 

 For the last claim, it follows from the equations in \eqref{eq:Ldtau}, \eqref{eq:Ltaula}, and \eqref{eq:Llatau} that $L$ and $\Lambda$ preserve
 $\Ke(\square)$. The last claim follows since $L^{n-k}: \cA^{p,k-p} \to \cA^{p+n-k,n-p}$ is injective in the given range, so the restriction is also injective, and therefore an isomorphism by Hodge duality. 
  \end{proof}

The arguments in fact show that the orthogonal decomposition and the dualities $(1)$, $(2)$, and $(3)$ of Theorem \ref{main} hold for the kernels of the operators  $\Delta_\del + \Delta_\delb$, $\Delta_\la + \Delta_\lab$, and $\Delta_\tau + \Delta_\taub$. But the Lefschetz-type properties appears to hold only for the operator $\square$. 

We now state some Corollaries of Theorem \ref{main}. The first follows from some well established facts about $\mathfrak{sl}(2,\CC)$ representations.

\begin{cor} \label{cor:nondec}
There is an orthogonal direct sum decomposition
\[
 \Ke(\square^{p,q})  = \bigoplus_{j \geq 0} L^j \left(  \Ke(\square^{p-j,q-j}) \right)_{\textrm{prim}}
\]
where 
\[
\left(  \Ke(\square^{r,s}) \right)_{\textrm{prim}} := \Ke(\square^{r,s})  \cap \Ke \Lambda.
\]
Moreover, 
\[
\dim  \Ke(\square^{p,q})  \leq \dim  \Ke(\square^{p+1,q+1}) \leq  \dots \leq \dim  \Ke(\square^{p+j,q+j})
\]
for all $0 \leq p, q \leq n$ and $p + q + 2j \leq n$, and the dimensions of the primitive spaces can be written in terms of successive differences of the numbers $\dim  \Ke(\square^{r,s})$. 
\end{cor}

The next corollary is in regards the Betti numbers $b^k:=\dim H^k(M)$ of a manifold $M$.
We first remark, it is immediate that 
$\Ke ( \Delta_\del + \Delta_\delb)  \subseteq \Ke (\Delta_d)$, 
and therefore for all $k\geq 0$ there are inequalities
\[\sum_{p+q=k}  \dim  \Ke \left( \Delta_\del^{p,q} + \Delta_\delb^{p,q} \right)   \leq b^{k}.
\]
Moreover,
\[
\dim  \Ke(  \Delta_\del^{p,q} + \Delta_\delb^{p,q})  \leq \min  \left( h^{p,q} , h^{q,p} \right),
\]
where $h^{p,q} = \dim \Ke (\Delta_\delb^{p,q})$ are the $(p,q)$-Hodge numbers.  Thus the same inequalities also hold for the kernel of $\square$ as well:
\begin{equation} \label{basicineq}
\sum_{p+q=k}  \dim  \Ke \left( \square^{p,q} \right)   \leq b^{k} \quad \quad \textrm{and} \quad \quad \dim  \Ke(  \square^{p,q})  \leq \min  \left( h^{p,q} , h^{q,p} \right).
\end{equation}
Thus, for example, for any Hermitian structure on $S^1 \times S^{2n-1}$, with $n > 1$, it holds that $\Ke(  \Delta_\del^{p,q} + \Delta_\delb^{p,q}) =\Ke (\square^{p,q}) =0$ for all $p,q$. 

The following much stronger Corollary requires the Lefschetz decomposition.

\begin{cor} \label{bettiineq}
If a compact Hermitian manifold of complex dimension $n$ has $b^k=0$  for some $k \leq n$, then 
\[
\Ke \left( \square^{p,q} \right) = 0
\]
for all $p,q$ satisfying $p+q \leq k$ and $p+q \equiv k \mod 2$, i.e. $p+q =k,k-2,k-4,\ldots$

Conversely, if 
\[
\Ke \left( \square^{p,q} \right) \neq 0
\]
for some $p,q$, then all of the Betti numbers $b^k$ must be non-zero for all $k$ satisfying 
$p+q \leq k \leq 2n-p-q$ and $p+q \equiv k \mod 2$.
\end{cor}

For example, for any Hermitian structure on $S^2 \times S^2 \times S^2$, we must have $\Ke \left( \square^{p,q} \right) = 0$ for all $p+q$ odd since $b^3=0$.

On the other hand, there are many non-trivial examples. 
One can compute the dimensions of the spaces $\Ke (\square^{p,q})$, for $\square$ acting on left invariant forms of the Kodaira-Thurston nilmanifold, with non-trivial lie bracket $[X,Y]=-Z$, complex structure given by $JX=Y$ and $JZ=W$, and Hermitian metric making $\{X,Y,Z,W\}$ orthonormal, to obtain
\[
\dim \left( \Ke  \square^{p,q}  \right) =
\arraycolsep=6pt\def\arraystretch{1.4}
 \begin{array}{|c|c|c|c|c|}
 \hline
 1&1&0  \\\hline
 1& 2 &1    \\\hline
 0&1&1 \\\hline
\end{array}
\quad \quad \textrm{whereas} \quad \quad
h^{p,q}= 
\arraycolsep=6pt\def\arraystretch{1.4}
 \begin{array}{|c|c|c|c|c|}
 \hline
 1&1&1 \\\hline
 2& 2 &2    \\\hline
 1&1&1 \\\hline
\end{array}.
\]

Another example is the Iwasawa manifold, a nilmanifold whose only nonzero differential on a left invariant basis of $(1,0)$-forms $X_1, X_2, X_3$ is $\del X_3 = -X_1 \wedge X_2$. With the Hermitian structure making $X_1, X_2, X_3$ orthonormal, so that $\omega = \sum_j X_j \overline{X_j}$, we have
\[
\dim \left( \Ke  \square^{p,q}  \right) =
\arraycolsep=6pt\def\arraystretch{1.4}
 \begin{array}{|c|c|c|c|c|}
 \hline
 1& 2 & 0  & 0\\\hline
 2& 2 & 2 & 0    \\\hline
 0& 2 & 2 & 2 \\\hline
 0& 0 & 2 &1  \\\hline
\end{array}
\quad \quad \textrm{whereas} \quad \quad
h^{p,q}=
\arraycolsep=6pt\def\arraystretch{1.4}
 \begin{array}{|c|c|c|c|c|}
 \hline
 1& 3 & 3  & 1\\\hline
 2& 6 & 6 & 2    \\\hline
 2& 6 & 6 & 2 \\\hline
 1& 3 & 3 &1  \\\hline
\end{array}.
\]

We give one more corollary which shows that the Hodge and Betti numbers of compact complex manifolds can be severely restricted by the algebra of the fundamental form and its derivatives.

\begin{cor}
If a holomorphic $p$-form $\eta$ on a compact Hermitian manifold $(M, \omega)$ is $\del$-harmonic, and satisfies $d \omega \wedge \eta = 0$, then all of the 
Hodge numbers $h^{p,0}, h^{p+1,1}, \ldots, h^{n,n-p}$ and $h^{0,p}, h^{1,p+1}, \ldots, h^{n-p,n}$ are non-zero. In particular, if $p \neq 0$, then all of the Betti numbers $b^p, b^{p+2}, \ldots , b^{2n-p}$ are at least two.
\end{cor}

\begin{proof}
By assumption $\eta \in \Ke (\Delta_\del + \Delta_\delb)$ and $\eta \in \Ke ( \la) \cap \Ke (\lab)$, and we claim that  $\eta$ is in the kernel of the zero-order summand $\Delta_\tau + \Delta_\taub + \Delta_\la + \Delta_\lab$. Indeed, for degree reasons,  $\eta$ is in the kernel of $\La, \la^*,  \lab^*$, and  $\taub^*$, which implies $\tau \eta = \taub \eta = 0$, and finally we use the fact that $[\La, \taub ] = 2i \tau^*$ to see $\eta$ is in the kernel of $\tau^*$. This shows $\square^{p,0} \eta = 0$, and the claims now follow from Corollary \ref{cor:nondec}, Equation \ref{basicineq}, and Corollary  \ref{bettiineq}.
\end{proof}

We conclude with some additional facts concerning the calculation of these harmonics spaces.
First, the definition of these harmonics simplifies in the case that the metric is pluri-closed (also known as strongly K\"ahler torsion, or SKT), i.e. $\del \delb \omega=0$. We use Demailly's Bochner-Kodaira-Nakano-type identity proved in \cite{De}:
\begin{equation} \label{eq:BKN}
\Delta_\delb = \Delta_{\del + \tau} + T_\omega,
\end{equation}
where $T_\omega = [\La, [\La, \frac i 2 \del \delb \omega]] - \Delta_\la$. By \eqref{eq:BKN},
\[
\Delta_{\del + \tau} = \Delta_\delb + \Delta_\la, \quad \textrm{and} \quad 
\Delta_{\delb + \taub} = \Delta_\del + \Delta_\lab,
\]
and therefore the following containments are equalities
\begin{align*}
\Ke (\Delta_\del + \Delta_\delb + \Delta_\tau + \Delta_\taub) &
\subseteq \Ke (\Delta_{\del+\tau} + \Delta_{\delb +\taub} + \Delta_\tau + \Delta_\taub) \\
& = \Ke (\square) \\
& \subseteq
 \Ke (\Delta_\del + \Delta_\delb + \Delta_\tau + \Delta_\taub ).
\end{align*}
This shows:
\begin{prop}
If the metric is pluri-closed, i.e. $\del \delb \omega = 0$, then
\[
\Ke (\square )=  \Ke (\Delta_\del + \Delta_\delb + \Delta_\tau + \Delta_\taub).
\]
\end{prop}

Next, focusing on $\la$ rather than $\tau$, we obtain some vanishing results by cohomological methods.  The square-zero operators $\la$ and $\lab$ induce cohomologies themselves. For each $q\geq0$, let
\[
H_\la^{p,q} (M,J ,\omega) = \frac{\Ke \left( \la: \cA^{p,q} \to \cA^{p+2,q+1} \right)} {\Img \left( \la: \cA^{p-2,q-1} \to \cA^{p,q} \right)},
\]
with induced differential $\del$. In the K\"ahler case, $H_\la^{p,q} = \cA^{p,q}$ for all $p,q$.
Since $\la$ is zero-order, and square zero, there is a fiberwise orthogonal Hodge decomposition
\[
\cA = \Img (\la) \oplus \left( \Ke \Delta_\la \right) \oplus \Img (\la^*)
\]
respecting the $(p,q)$ bigrading. It follows that 
$H_\la^{p,q} \cong \Ke \Delta_\la^{p,q}$, and moreover the groups satisfy Serre duality, since  $\Delta_\la \star = - \star \Delta_{\overline{\la}}$. Analogous results hold for $\lab$, and we conclude:

\begin{prop}
There is an inclusion $\Ke (\square^{p,q}) \to H_\la^{p,q}$,  so the vanishing of the latter zero-order cohomology implies the second order elliptic operator has zero kernel. Similarly, this holds for $H_\lab^{p,q}$.
\end{prop}

This includes some interesting special cases, including Hermitian manifolds that are particularly \emph{far} from being K\"ahler. For example, if the map 
 $\lab :\cA^{1,0} \to \cA^{2,2}$ is everywhere injective in complex dimension $3$ or more, then $\Ke (\square^{1,0})$ is zero.
   
It is particular to complex dimension $3$ that the map $\lab :\cA^{2,0} \to \cA^{3,2}$ can be an isomorphism, and in this case we conclude $\Ke (\square^{p,q}) = 0$ for $(p,q) = (2,0), (3,2), (1,3),$ and $(0,1)$. It would be interesting to study this class of highly non-symplectic metrics on (almost) complex manifolds. The mirror concept in complex dimension $3$, when the component $\mub:\cA^{1,0} \to \cA^{0,2}$ of the exterior derivative is an isomorphism, pertains to 
 maximally non-integrable almost complex structures, c.f. \cite{CWDol}, which includes the family of strictly nearly K\"ahler $6$-manifolds.

On the other hand, while the operator $\tau$ is not a differential, instead of a cohomological vanishing argument, one can rely on an analytic argument. 
 
 \begin{prop}
 If any of the zero-order operators $\tau,\taub,\la, \lab$ or their adjoints are injective at a single point, on bidegree $(p,q)$, then $  \Ke(\square^{p,q})=0$. 
 \end{prop}
 
 In particular, $\dim  \Ke(\square^{0,0})$ is either one or zero, depending on whether the Hermitian structure is K\"ahler.
 
 \begin{proof}
If $\eta \in \Ke(\square^{p,q})$, and $\delta$ is any such operator in the claim, then $\delta \eta =0$ at a point implies $\eta$ vanishes on an open set, and therefore the $d$-harmonic form $\eta$ is identically zero. 
\end{proof}

It remains to further study what these numbers tell us about a given Hermitian structure, and conversely, to determine what are the permissible numbers for a given complex structure.

\bibliographystyle{alpha}

\bibliography{biblio}

\end{document}